\documentclass[11pt]{article}
\usepackage{jhtitle,amsbsy,amsmath,amsthm,amssymb,times,graphicx}
\usepackage{probmacros}
\usepackage[sort]{natbib}

\usepackage{comment}
\specialcomment{details}{\begin{quotation}\small}{\end{quotation}}
\excludecomment{details}

\def\baro{\vskip  .2truecm\hfill \hrule height.5pt \vskip  .2truecm}
\def\barba{\vskip -.1truecm\hfill \hrule height.5pt \vskip .4truecm}

\newcommand{\pcite}[1]{\citeauthor{#1}'s \citeyearpar{#1}}

\newcommand{\sumn}{\sum_{i=1}^{n}}

\newtheorem{proposition}{Proposition}
\newtheorem{corollary}{Corollary}
\newtheorem{lemma}{Lemma}

\newtheorem{theorem}{Theorem}
\newtheorem{remark}{Remark}

\newcommand{\X}{{\mathsf{X}}}

\begin{document}

\title{Trace-class Monte Carlo Markov Chains for Bayesian Multivariate
  Linear Regression with Non-Gaussian Errors} \author{Qian Qin and
  James P. Hobert \\ Department of Statistics \\ University of
  Florida} \date{January 2016}


\maketitle

\begin{abstract}
  Let $\pi$ denote the intractable posterior density that results when
  the likelihood from a multivariate linear regression model with
  errors from a scale mixture of normals is combined with the standard
  non-informative prior.  There is a simple data augmentation
  algorithm (based on latent data from the mixing density) that can be
  used to explore $\pi$.  Let $h(\cdot)$ and $d$ denote the mixing
  density and the dimension of the regression model, respectively.
  \cite{hobe:jung:khar:qin:2016} have recently shown that, if $h$
  converges to 0 at the origin at an appropriate rate, and
  $\int_0^\infty u^{\frac{d}{2}} \, h(u) \, du < \infty$, then the
  Markov chains underlying the DA algorithm and an alternative Haar
  PX-DA algorithm are both geometrically ergodic.  In fact, something
  much stronger than geometric ergodicity often holds.  Indeed, it is
  shown in this paper that, under simple conditions on $h$, the Markov
  operators defined by the DA and Haar PX-DA Markov chains are
  \textit{trace-class}, i.e., compact with summable eigenvalues.  Many
  of the mixing densities that satisfy \pcite{hobe:jung:khar:qin:2016}
  conditions also satisfy the new conditions developed in this paper.
  Thus, for this set of mixing densities, the new results provide a
  substantial strengthening of \pcite{hobe:jung:khar:qin:2016}
  conclusion without any additional assumptions.  For example,
  \cite{hobe:jung:khar:qin:2016} showed that the DA and Haar PX-DA
  Markov chains are geometrically ergodic whenever the mixing density
  is generalized inverse Gaussian, log-normal, Fr\'{e}chet (with shape
  parameter larger than $d/2$), or inverted gamma (with shape
  parameter larger than $d/2$).  The results in this paper show that,
  in each of these cases, the DA and Haar PX-DA Markov operators are,
  in fact, trace-class.
\end{abstract}

\section{Introduction}
\label{sec:intro}

Consider the multivariate linear regression model
\begin{equation}
  \label{eq:mreg2}
  Y = X \beta + \varepsilon \Sigma^{\frac{1}{2}} \;,
\end{equation}
where $Y$ denotes an $n \times d$ matrix of responses, $X$ is an $n
\times p$ matrix of known covariates, $\beta$ is a $p \times d$ matrix
of unknown regression coefficients, $\Sigma^{\frac{1}{2}}$ is an
unknown positive-definite scale matrix, and $\varepsilon$ is an $n
\times d$ matrix whose rows are iid random vectors from a scale
mixture of multivariate normal densities.  In particular, letting
$\varepsilon_i^T$ denote the $i$th row of $\varepsilon$, we assume
that $\varepsilon_i$ has density
\begin{equation*}
  f_h(\varepsilon) = \int_{0}^{\infty} \frac{u^{\frac{d}{2}}}{(2
    \pi)^{\frac{d}{2}}} \, \exp \Big \{ -\frac{u}{2}
  \varepsilon^T \varepsilon \Big \} h(u) \, du \;,
\end{equation*}
where $h: (0,\infty) \rightarrow [0,\infty)$ is the so-called
\textit{mixing density}.  Error densities of this form are often used
when heavy-tailed errors are required.  For example, it is well known
that if $h$ is a $\mbox{Gamma}(\frac{\nu}{2}, \frac{\nu}{2})$ density
(with mean 1), then $f_h$ becomes the multivariate Student's $t$
density with $\nu$ degrees of freedom.

A Bayesian analysis of the data from this regression model requires a
prior on $(\beta,\Sigma)$.  We consider an improper default prior that
takes the form $\omega (\beta , \Sigma) \propto |\Sigma|^{-a} \,
I_{{\cal S}_d}(\Sigma)$ where ${\cal S}_d \subset
\mathbb{R}^{\frac{d(d + 1)}{2}}$ denotes the space of $d \times d$
positive definite matrices.  Taking $a=(d+1)/2$ yields the
independence Jeffreys prior, which is the standard non-informative
prior for multivariate location scale problems.  Of course, whenever
an improper prior is used, one must check that the corresponding
posterior distribution is proper.  Letting $y$ denote the observed
value of $Y$, the joint density of the data from model
\eqref{eq:mreg2} can be expressed as
\begin{equation*}
  f (y | \beta, \Sigma) = \prod_{i=1}^n \Bigg[ \int_{0}^{\infty}
  \frac{u^{\frac{d}{2}}}{(2\pi)^{\frac{d}{2}} |\Sigma|^{\frac{1}{2}}}
  \exp \bigg\{ -\frac{u}{2} \Big(y_i - \beta^T x_i\Big)^T
  \Sigma^{-1}\Big(y_i - \beta^T x_i\Big) \bigg\} h(u) \, du \Bigg] \;.
\end{equation*}
Define
\[
m(y) = \int_{{\cal S}_d} \int_{\mathbb{R}^{p \times d}}
f(y|\beta,\Sigma) \, \omega(\beta,\Sigma) \, d\beta \, d\Sigma \;.
\]
The posterior distribution is proper precisely when $m(y) < \infty$.
Let $\Lambda$ stand for the $n \times (p+d)$ matrix $(X:y)$.
Straightforward arguments (using ideas from \citet{fern:stee:1999})
show that, together, the following four conditions are sufficient for
posterior propriety:
\begin{enumerate}
  \item[($S1$)] $\mbox{rank}(\Lambda) = p+d \;;$
  \item[($S2$)] $n > p + 2d - 2a \;;$
  \item[($S3$)] $\int_0^\infty u^{\frac{d}{2}} \, h(u) \, du < \infty \;;$
  \item[($S4$)] $\int_0^\infty u^{-\frac{n-p+2a-2d-1}{2}} \, h(u) \,
    du < \infty \;.$
\end{enumerate}
These four conditions are assumed to hold throughout this paper.

\begin{remark}
  Conditions ($S1$) \& ($S2$) are known to be necessary for posterior
  propriety \citep{fern:stee:1999,hobe:jung:khar:qin:2016}.
\end{remark}

\begin{remark}
  Condition ($S3$) clearly concerns the tail behavior of $h$.
  Similarly, condition ($S4$) concerns the behavior of $h$ near the
  origin, unless $n-p+2a-2d-1$ is negative, which is possible.  Note,
  however, that ($S2$) implies that $-(n-p+2a-2d-1)/2 < 1/2$.
  Consequently, if $n-p+2a-2d-1$ is negative, then ($S4$) is implied
  by ($S3$).
\end{remark}

Of course, the posterior density of $(\beta,\Sigma)$ given the data
takes the form
\[
\pi(\beta, \Sigma|y) = \frac{f(y|\beta,\Sigma) \,
  \omega(\beta,\Sigma)}{m(y)} \;.
\]
There is a well-known data augmentation (DA) algorithm that can be
used to explore this intractable density \citep{liu:1996}.
\citet{hobe:jung:khar:qin:2016} (hereafter HJK\&Q) performed
convergence rate analyses of the Markov chains underlying this DA
algorithm and an alternative Haar PX-DA algorithm.  In this paper, we
provide a substantial improvement of HJK\&Q's main result.  A formal
statement of the DA algorithm requires some buildup.  Let $z =
(z_1,\dots,z_n)$ have strictly positive elements, and let $Q=Q(z)$ be
the $n \times n$ diagonal matrix whose $i$th diagonal element is
$z_i^{-1}$.  Also, define $\Omega = (X^T Q^{-1} X)^{-1}$ and $\mu =
(X^T Q^{-1} X)^{-1} X^T Q^{-1} y$.  For each $s \ge 0$, define a
univariate density as follows
\begin{equation}
  \label{eq:psi}
  \psi(u;s) = b(s) \, u^{\frac{d}{2}} \, e^{ -\frac{s u}{2}} \, h(u) \;,
\end{equation}
where $b(s)$ is the normalizing constant.  The DA algorithm uses draws
from the inverse Wishart ($\mbox{IW}_d$) and matrix normal
($\mbox{N}_{p,d}$) distributions.  These densities are defined in the
Appendix.  If the current state of the DA Markov chain is
$(\beta_m,\Sigma_m) = (\beta,\Sigma)$, then we simulate the new state,
$(\beta_{m+1},\Sigma_{m+1})$, using the following three-step
procedure.

\baro \vspace*{2mm}
\noindent {\rm Iteration $m+1$ of the DA algorithm:}
\begin{enumerate}
\item Draw $\{Z_i\}_{i=1}^n$ independently with $Z_i \sim \psi \Big(
  \cdot \;; \big( \beta^T x_i - y_i \big)^T \Sigma^{-1} \big( \beta^T
  x_i - y_i \big) \Big)$, and call the result $z = (z_1,\dots,z_n)$.
\item Draw
\[
\Sigma_{m+1} \sim \mbox{IW}_d \bigg( n-p+2a-d-1, \Big( y^T Q^{-1} y -
  \mu^T \Omega^{-1} \mu \Big)^{-1} \bigg)  \;.
\]
\item Draw $\beta_{m+1} \sim \mbox{N}_{p,d} \big( \mu, \Omega,
  \Sigma_{m+1} \big)$ \vspace*{-2.5mm}
\end{enumerate}
\barba

\noindent Denote the DA Markov chain by $\Phi =
\{(\beta_m,\Sigma_m)\}_{m=0}^\infty$, and its state space by $\X :=
\mathbb{R}^{p \times d} \times {\cal S}_d$.  For positive integer $m$,
let $k^m: \X \times \X \rightarrow (0,\infty)$ denote the $m$-step
Markov transition density (Mtd) of $\Phi$, so that if $A$ is a
measurable set in $\X$,
\[
P \Big((\beta_m,\Sigma_m) \in A \, \big| \, (\beta_0,\Sigma_0) =
(\beta,\Sigma) \Big) = \int_A k^m\big( (\beta',\Sigma') \big|
(\beta,\Sigma) \big) \, d\beta' \, d\Sigma' \;.
\]
The 1-step Mtd, $k \equiv k^1$, can be expressed as
\[
k\big( (\beta',\Sigma') \big| (\beta,\Sigma) \big) =
\int_{\mathbb{R}^n_+} \pi(\beta'|\Sigma',z,y) \, \pi(\Sigma'|z,y) \,
\pi(z|\beta,\Sigma,y) \, dz \;,
\]
where the precise forms of the conditional densities
$\pi(z|\beta,\Sigma,y)$, $\pi(\Sigma|z,y)$, and
$\pi(\beta|\Sigma,z,y)$ can be gleaned from steps 1., 2., and 3. of
the DA algorithm, respectively.  If there exist $M: \X \rightarrow
[0,\infty)$ and $\lambda \in [0,1)$ such that, for all $m$,
\begin{equation*}
  \int_{{\cal S}_d} \int_{\mathbb{R}^{p \times d}} \Big| k^m \big(
  \beta,\Sigma \big| \tilde{\beta},\tilde{\Sigma} \big) -
  \pi(\beta,\Sigma \big| y) \Big| \, d\beta \, d\Sigma \le
  M(\tilde{\beta},\tilde{\Sigma}) \, \lambda^m \;,
\end{equation*}
then the chain $\Phi$ is \textit{geometrically ergodic}.  The benefits
of using a geometrically ergodic Monte Carlo Markov chain have been
well documented \citep[see,
e.g.][]{robe:rose:1998,jone:hobe:2001,fleg:hara:jone:2008}.  HJK\&Q
showed that, if $h$ converges to zero at the origin at an appropriate
rate, then $\Phi$ is geometrically ergodic.  In order to state
HJK\&Q's result, we must introduce three classes of mixing densities.
Let $h: (0,\infty) \rightarrow [0,\infty)$ be a mixing density.  If
there is an $\eta>0$ such that $h(u)=0$ for all $u \in (0,\eta)$, then
we say that $h$ is \textit{zero near the origin}.  Now assume that $h$
is strictly positive in a neighborhood of 0.  If there exists a $c>-1$
such that
\[
\lim_{u \rightarrow 0} \frac{h(u)}{u^c} \in (0,\infty) \;,
\]
then we say that $h$ is \textit{polynomial near the origin} with power
$c$.  Finally, if for every $c>0$, there exists an $\eta_c>0$ such
that the ratio $\frac{h(u)}{u^c}$ is strictly increasing in
$(0,\eta_c)$, then we say that $h$ is \textit{faster than polynomial
  near the origin}.  HJK\&Q showed that every mixing density that is a
member of a standard parametric family is in one of these three
classes, and they proved the following result.

\begin{theorem}[HJK\&Q]
  \label{thm:qh}
  If the mixing density, $h$, is zero near the origin, or faster than
  polynomial near the origin, or polynomial near the origin with power
  $c > \frac{n-p+2a-d-1}{2}$, then the DA Markov chain is
  geometrically ergodic.
\end{theorem}

\begin{remark}
  It is not necessary to check that ($S4$) holds before applying
  Theorem~\ref{thm:qh} because, together with ($S3$), the hypothesis
  of Theorem~\ref{thm:qh} implies that ($S4$) is satisfied.
\end{remark}

In this paper, we show that something \textit{much stronger} than
geometric ergodicity often holds.  We begin with some requisite
background material on Markov operators.  The posterior density can be
used to define an inner product
\[
\langle f_1,f_2 \rangle = \int_\X f_1(\beta, \Sigma) \, f_2(\beta, \Sigma) \,
\pi(\beta, \Sigma|y) \, d\beta \, d\Sigma \;,
\]
and norm $\norm{f} = \sqrt{\langle f,f \rangle}$ on the Hilbert space
\[
L_0^2 = \bigg \{ f: \X \rightarrow \mathbb{R} : \int_\X f^2(\beta,
\Sigma) \, \pi(\beta, \Sigma|y) \, d\beta \, d\Sigma < \infty
\hspace*{2mm} \mbox{and} \hspace*{2mm} \int_\X f(\beta, \Sigma) \,
\pi(\beta, \Sigma|y) \, d\beta \, d\Sigma = 0 \bigg \} \;.
\]
Now define the DA Markov operator $K : L_0^2 \rightarrow L_0^2$ as
that which takes $f \in L_0^2$ into
\[
(Kf)(\beta, \Sigma) = \int_\X f(\beta', \Sigma') \, k\big(
(\beta',\Sigma') \big| (\beta,\Sigma) \big) \, d\beta' \, d\Sigma' \;.
\]
Because $K$ is based on a DA algorithm, it is self-adjoint and
positive \citep{liu:wong:kong:1994}.  If, in addition, $K$ is also a
compact operator, then $K$ has a pure eigenvalue spectrum, all of its
eigenvalues reside in $[0,1)$, and the corresponding Markov chain is
geometrically ergodic \citep[see,
e.g.,][]{mira:geye:1999,hobe:roy:robe:2011}.  We note that the set of
Monte Carlo Markov chains whose operators are compact is a small
subset of those that are geometrically ergodic \citep[see,
e.g.,][p. 1755]{chan:geye:1994}.  Taking this a step further, $K$ is
said to be \textit{trace-class} if it is compact \textit{and} its
eigenvalues are summable \citep[see, e.g.][p. 267]{conw:1990}.  In
this paper, we provide sufficient conditions (on $h$) for $K$ to be
trace-class.  The benefits of using trace-class Markov operators are
spelled out in \citet{khar:hobe:2011}, and we exploit their results in
Section~\ref{sec:Haar}.

A statement of our main result requires substantial build-up, so here
in the Introduction we present only one simple, but powerful,
corollary.  Let $\mathbb{R}_+$ denote the set $(0,\infty)$, and define
a parametric family of functions $g_{\rho,\tau} : \mathbb{R}_+
\rightarrow [0,\infty)$ as follows.  For $\rho \in \mathbb{R}_+$ and
$\tau \in \mathbb{R}$, let
\[
g_{\rho,\tau}(u) = \exp \big\{ -\rho(\log u)^2 + \tau \log u \big\} \;.
\]
The following result is a corollary of Theorem~\ref{thm:main} in
Section~\ref{sec:main}.

\begin{corollary}
  \label{cor:main}
  Let $h$ be a mixing density.  If there exist $\rho \in
  \mathbb{R}_+$, $\tau \in \mathbb{R}$ and $\eta>0$ such that
  $h(u)/g_{\rho,\tau}(u)$ is non-decreasing in $(0,\eta)$, then the DA
  Markov operator, $K$, is trace-class.
\end{corollary}

An immediate consequence of Corollary~\ref{cor:main} is that, if $h$
is zero near the origin, then $K$ is trace-class.  Indeed, for any
$(\rho,\tau) \in \mathbb{R}_+ \times \mathbb{R}$,
$h(u)/g_{\rho,\tau}(u)$ is constant (and equal to zero) in a
neighborhood of the origin.  Corollary~\ref{cor:main} also implies
that if $h$ is a member of one of the standard parametric families
that are faster than polynomial near the origin (inverted gamma,
log-normal, generalized inverse Gaussian, and Fr\'{e}chet), then the
corresponding Markov operator is trace-class.  For example, consider
the case where the mixing density is inverted gamma.  In particular,
let $h(u) = b \, u^{-\alpha-1} e^{-\gamma/u} I_{\mathbb{R}_+}(u)$,
where $\alpha>d/2$, $\gamma>0$ and $b=b(\alpha,\gamma)$ is the
normalizing constant.  (We require $\alpha > d/2$ so that condition
($S3$) is satisfied.)  Taking $\rho=1$ and $\tau=-(\alpha+1)$, we have
\[
\frac{d}{du} \frac{h(u)}{g_{\rho,\tau}(u)} = b \, \frac{d}{du} \exp
\Big\{ -\frac{\gamma}{u} + (\log u)^2 \Big\} = \frac{b}{u} \Big[
\frac{\gamma}{u} + 2 \log u \Big] \exp \Big\{ -\frac{\gamma}{u} +
(\log u)^2 \Big\} \;,
\]
which is clearly positive in a neighborhood of 0.  Hence,
Corollary~\ref{cor:main} implies that $K$ is trace-class.  (This
result was established by \citet{jung:hobe:2014} in the special case
where $d=1$.)  Similar arguments can be used for the other three
families (log-normal, generalized inverse Gaussian and Fr\'{e}chet),
and these are given in Section~\ref{sec:main}.  Indeed, for a large
class of mixing densities (including the ones just mentioned) our
results provide a substantial strengthening of
\pcite{hobe:jung:khar:qin:2016} conclusion \textit{without any
  additional assumptions}.  On the other hand, as we now explain,
there are still many mixing densities that satisfy the hypotheses of
Theorem~\ref{thm:qh}, but to which our results are not applicable.

The following lemma, which is proven in Section~\ref{sec:main},
provides a sufficient condition for $K$ to be trace-class, and is one
of the key pieces of the proof of Theorem~\ref{thm:main} (and hence of
Corollary~\ref{cor:main}).

\vspace*{3mm}
\noindent
{\bf Lemma 2.\;}
\textit{Let $h$ be a mixing density that is strictly positive in a
  neighborhood of the origin.  If there exist $\zeta \in (1,2)$ and
  $\eta>0$ such that
\begin{equation}
  \label{eq:suff2}
  \int_{0}^{\eta} \frac{u^{\frac{d}{2}} h(u)}{\int_0^{\zeta u}
    v^{\frac{d}{2}} h(v) \, dv} \, du < \infty \;,
\end{equation}
then $K$ is trace-class.}
\vspace*{3mm}

Note that \eqref{eq:suff2} cannot hold if, for each $\eta>0$,
\begin{equation}
  \label{eq:fail}
\int_{0}^{\eta} \frac{u^{\frac{d}{2}} h(u)}{\int_0^{2u}
  v^{\frac{d}{2}} h(v) \, dv} \, du = \infty \;.
\end{equation}
Consequently, when $h$ is strictly positive in a neighborhood of the
origin and \eqref{eq:fail} holds, our results cannot be applied to
$h$.  For example, assume that $h$ is polynomial near the origin with
power $c$ so that $\frac{h(u)}{u^c} \rightarrow l \in \mathbb{R}_+$ as
$u \rightarrow 0$.  Then for all $u$ is some small neighborhood of the
origin, we have $l/2<\frac{h(u)}{u^c}<2l$.  So, if $\eta>0$ is small
enough, then for all $u \in (0,\eta)$,
\[
\frac{u^{\frac{d}{2}} h(u)}{\int_0^{2u} v^{\frac{d}{2}} h(v) \, dv}
\ge \frac{u^{\frac{d}{2}} \frac{l}{2} u^c}{\int_0^{2u} v^{\frac{d}{2}}
  2 l v^c \, dv} = \frac{b}{u} \;,
\]
where $b=b(d,l)$ is a positive constant.  Thus, since $\int_0^\eta
\frac{1}{u} \, du$ diverges for every $\eta>0$, \eqref{eq:fail} holds.
Consequently, our results are not applicable to mixing densities that
are polynomial near the origin.  Furthermore, in
Section~\ref{sec:main}, we give an example of a mixing density that is
faster than polynomial near the origin, but for which \eqref{eq:fail}
holds.

The remainder of this paper is organized as follows.  The main result
is stated and proven in Section~\ref{sec:main}.  In
Section~\ref{sec:ftp}, we examine the consequences of the main result
when the mixing density is faster than polynomial near the origin.  In
Section~\ref{sec:Haar}, we show that Theorem~\ref{thm:main} has
important implications for a Haar PX-DA variant of the DA algorithm
that was introduced by \citet{roy:hobe:2010} and extended by HJK\&Q.
Finally, the Appendix contains the definitions of the $\mbox{IW}_d$
and $\mbox{N}_{p,d}$ families, as well as some technical details.

\section{Main Result}
\label{sec:main}

In this section, we will be dealing with functions $g: \mathbb{R}_+
\rightarrow [0,\infty)$ that are strictly positive and differentiable
in a neighborhood of the origin.  Let $\mathcal{A}$ denote the set of
all such functions, and let $\mathcal{K}$ denote the subset of
$\mathcal{A}$ consisting of functions whose reciprocals are integrable
near the origin, i.e.,
\[
\mathcal{K} = \Big \{ \kappa \in \mathcal{A} : \int_0^\eta
\frac{1}{\kappa(u)} \, du < \infty \; \mbox{for some $\eta>0$} \Big \}
\;.
\]
The function $\kappa(u) = u (\log u)^2$ is a member of $\mathcal{K}$,
and we will use this fact in the sequel.  Now, for fixed $\kappa \in
\mathcal{K}$ and fixed $\zeta \in (1,2)$, let $\mathcal{C}(\kappa,\zeta)$
denote the subset of $\mathcal{A}$ containing the functions $g$ that
satisfy the following three conditions:
\begin{enumerate}
\item $u^{\frac{d}{2}} g(u)$ is bounded in a neighborhood of the
  origin,
\item $\lim_{u \rightarrow 0} \kappa(u) u^{\frac{d}{2}} g(u) = 0$,
\item There exist $l_1,l_2 \in \mathbb{R}$ such that
\begin{equation}
    \label{eq:twocs}
    \lim_{u \rightarrow 0} \Big( \kappa'(u) + \frac{d}{2}
    \frac{\kappa(u)}{u} \Big) \frac{g(u)}{g(\zeta u)} = l_1 \hspace*{6mm}
    \mbox{and} \hspace*{6mm} \lim_{u \rightarrow 0} \frac{\kappa(u) g'(u)}
    {g(\zeta u)} =l_2 \;.
\end{equation}
\end{enumerate}
The following result is proven in the Appendix.

\begin{proposition}
  \label{prop:pf}
  Fix $(\rho,\tau) \in \mathbb{R}_+ \times \mathbb{R}$.  Then
  $g_{\rho,\tau} \in \mathcal{C}(\kappa,3/2)$, with $\kappa(u) = u
  (\log u)^2$.  Furthermore, $u^{\frac{d}{2}} g_{\rho,\tau}(u)$ is
  non-decreasing in a neighborhood of the origin.
\end{proposition}
Here is our main result.
\begin{theorem}
  \label{thm:main}
  Let $h$ be a mixing density.  Each of the following three conditions
  is sufficient for the corresponding DA Markov operator, $K$, to be
  trace-class.
\begin{enumerate}
\item The mixing density $h$ is zero near the origin.
\item There exist $\kappa \in \mathcal{K}$, $\zeta \in (1,2)$ and $g \in
  \mathcal{C}(\kappa,\zeta)$ such that $\lim_{u \rightarrow 0}
  \frac{h(u)}{g(u)} \in \mathbb{R}_+$.
\item There exist $\kappa \in \mathcal{K}$, $\zeta \in (1,2)$ and $g \in
  \mathcal{C}(\kappa,\zeta)$ such that both $u^{\frac{d}{2}} g(u)$ and
  $\frac{h(u)}{g(u)}$ are non-decreasing in a neighborhood of the
  origin.
\end{enumerate}
\end{theorem}

\begin{remark}
  Suppose that $h \in \mathcal{C}(\kappa,\zeta)$.  Then, by taking $g=h$,
  the second condition of Theorem~\ref{thm:main} is satisfied, so $K$
  is trace-class.  However, this argument requires that $h$ be
  differentiable in a neighborhood of the origin.  The surrogate
  function, $g$, allows us to handle non-differentiable mixing
  densities.
\end{remark}

\begin{remark}
  Note that Corollary~\ref{cor:main} (from the Introduction) follows
  immediately from Theorem~\ref{thm:main} and
  Proposition~\ref{prop:pf}.
\end{remark}

Our proof of Theorem~\ref{thm:main} is based on three lemmas, which we
now state and prove.

\begin{lemma}
  \label{lem:suff}
  Let $h$ be a mixing density, and let $\psi(u;s)$ be as in
  \eqref{eq:psi}.  Suppose there exist $\zeta<2$ and $\nu: \mathbb{R}_+
  \rightarrow [0,\infty)$ with $\int_{\mathbb{R}_+} \nu(u) \, du <
  \infty$ such that
\begin{equation}
  \label{eq:main}
  \psi(u;s) \le \exp \bigg \{ \frac{(\zeta-1)us}{2} \bigg\} \nu(u)
\end{equation}
for all $u \in \mathbb{R}_+$ and all $s \in [0,\infty)$.  Then $K$ is
trace-class.
\end{lemma}

\begin{proof}
  For $i=1,2,\dots,n$, define $r_i = r_i(\beta,\Sigma) = \big( \beta^T
  x_i - y_i \big)^T \Sigma^{-1} \big( \beta^T x_i - y_i \big)$.  Of
  course, $r_i \ge 0$.  First, it suffices to show that
\[
\int_{\mathcal{S}_d} \int_{\mathbb{R}^{p \times d}} k \big(
(\beta,\Sigma) \big| (\beta,\Sigma) \big) \, d\beta \, d\Sigma <
\infty \;,
\]
\citep[see, e.g.,][]{khar:hobe:2011}.  Routine calculations show that
\begin{align*}
  \pi(\beta,\Sigma \big| z,y) \pi(z \big| \beta,\Sigma,y) & =
  \frac{|\Sigma|^{-\frac{n+2a}{2}} \exp \big \{ -\frac{1}{2} \sumn r_i
    z_i \big\}}{\int_{\mathcal{S}_d} \int_{\mathbb{R}^{p \times d}}
    |\Sigma|^{-\frac{n+2a}{2}} \exp \big \{ -\frac{1}{2} \sumn r_i z_i
    \big\} \, d\beta \, d\Sigma} \prod_{i=1}^n \psi(z_i;r_i) \\ & \le
  \frac{|\Sigma|^{-\frac{n+2a}{2}} \exp \big \{ - \frac{(2-\zeta)}{2}
    \sumn r_i z_i \big\}}{\int_{\mathcal{S}_d} \int_{\mathbb{R}^{p
        \times d}} |\Sigma|^{-\frac{n+2a}{2}} \exp \big \{
    -\frac{1}{2} \sumn r_i z_i \big\} \, d\beta \, d\Sigma}
  \prod_{i=1}^n \nu(z_i) \;.
\end{align*}
The transformation $\Sigma' = \Sigma/(2-\zeta)$, yields
\begin{align*}
  \int_{\mathcal{S}_d} \int_{\mathbb{R}^{p \times d}}
  |\Sigma|^{-\frac{n+2a}{2}} \exp & \Big \{ - \frac{(2-\zeta)}{2} \sumn
  r_i z_i \Big\} \, d\beta \, d\Sigma \\ & =
  \frac{1}{(2-\zeta)^{\frac{(n+2a-d-1)d}{2}}} \int_{\mathcal{S}_d}
  \int_{\mathbb{R}^{p \times d}} |\Sigma|^{-\frac{n+2a}{2}} \exp \Big
  \{ -\frac{1}{2} \sumn r_i z_i \Big\} \, d\beta \, d\Sigma \;.
\end{align*}
It follows that,
\[
\int_{\mathcal{S}_d} \int_{\mathbb{R}^{p \times d}} \pi(\beta,\Sigma
\big| z,y) \pi(z \big| \beta,\Sigma,y) \, d\beta \, d\Sigma \le
\frac{1}{(2-\zeta)^{\frac{(n+2a-d-1)d}{2}}} \prod_{i=1}^n \nu(z_i) \;.
\]
Therefore,
\begin{align*}
  \int_{\mathcal{S}_d} \int_{\mathbb{R}^{p \times d}} k \big(
  (\beta,\Sigma) \big| (\beta,\Sigma) \big) \, d\beta \, d\Sigma & =
  \int_{\mathbb{R}^n_+} \int_{\mathcal{S}_d} \int_{\mathbb{R}^{p
      \times d}} \pi(\beta,\Sigma|z,y) \, \pi(z|\beta,\Sigma,y) \, d\beta
  \, d\Sigma \, dz \\ & \le \frac{1}{(2-\zeta)^{\frac{(n+2a-d-1)d}{2}}}
  \bigg( \int_{\mathbb{R}_+} \nu(u) \, du \bigg)^n < \infty \;.
\end{align*}
\end{proof}

The following lemma was given in the Introduction, and is restated
here for convenience.

\begin{lemma}
  \label{lem:suff2}
  Let $h$ be a mixing density that is strictly positive in a
  neighborhood of the origin.  If there exist $\zeta \in (1,2)$ and
  $\eta>0$ such that
\begin{equation}
  \int_{0}^{\eta} \frac{u^{\frac{d}{2}} h(u)}{\int_0^{\zeta u}
    v^{\frac{d}{2}} h(v) \, dv} \, du < \infty \tag{\ref{eq:suff2}} \;,
\end{equation}
then $K$ is trace-class.
\end{lemma}

\begin{proof}
  First, note that
\[
\psi(u;s) = \frac{u^{\frac{d}{2}} \, e^{ -\frac{s u}{2}} \,
  h(u)}{\int_{\mathbb{R}_+} v^{\frac{d}{2}} \, e^{ -\frac{s v}{2}} \,
  h(v) \, dv} \le \frac{u^{\frac{d}{2}} \, e^{ -\frac{s u}{2}} \,
  h(u)}{\int_0^{\zeta u} v^{\frac{d}{2}} \, e^{ -\frac{s v}{2}} \, h(v) \,
  dv} \le \frac{\exp \Big \{ \frac{(\zeta-1)su}{2} \Big \} u^{\frac{d}{2}}
  \, h(u)}{\int_0^{\zeta u} v^{\frac{d}{2}} \, h(v) \, dv} \;.
\]
By Lemma~\ref{lem:suff}, it suffices to show that
\[
\int_{\mathbb{R}_+} \frac{u^{\frac{d}{2}} \, h(u)}{\int_0^{\zeta u}
  v^{\frac{d}{2}} \, h(v) \, dv} \, du < \infty \;.
\]
But, for any $\eta>0$, we have
\[
\int_{\eta}^\infty \frac{u^{\frac{d}{2}} \, h(u)}{\int_0^{\zeta u}
  v^{\frac{d}{2}} \, h(v) \, dv} \, du \le \frac{\int_{\eta}^\infty
  u^{\frac{d}{2}} \, h(u) \, du}{\int_0^{\zeta\eta} v^{\frac{d}{2}} \,
  h(v) \, dv} < \infty \;,
\]
and the result follows.
\end{proof}

\begin{lemma}
  \label{lem:suff3}
  Let $g \in \mathcal{C}(\kappa,\zeta)$ for some $\kappa \in \mathcal{K}$
  and some $\zeta \in (1,2)$.  By assumption, there exists $\eta_0>0$ such
  that $g$ is strictly positive and differentiable on $(0,\eta_0)$.
  Then for any $\eta \in (0,\eta_0)$, we have
  \begin{equation*}
    \int_{0}^{\eta} \frac{u^{\frac{d}{2}} g(u)}{\int_0^{\zeta u}
      v^{\frac{d}{2}} g(v) \, dv} \, du < \infty \;.
  \end{equation*}
\end{lemma}

\begin{proof}
  Since $u^{\frac{d}{2}} g(u)$ is bounded in a neighborhood of the
  origin, we have
\[
\lim_{u \rightarrow 0} \int_0^{\zeta u} v^{\frac{d}{2}} g(v) \, dv = 0 \;.
\]
Hence, an application of L'H\^{o}pital's rule yields
\begin{align}
   \label{eq:L1}
   \lim_{u \rightarrow 0} \frac{\kappa(u)
     u^{\frac{d}{2}}g(u)}{\int_0^{\zeta u} v^{\frac{d}{2}} g(v) \, dv} & =
   \lim_{u \rightarrow 0} \frac{\big[ \kappa'(u) u^{\frac{d}{2}} +
     \kappa(u) \frac{d}{2} u^{\frac{d}{2}-1} \big]g(u) + \kappa(u)
     u^{\frac{d}{2}} g'(u)}{\zeta(\zeta u)^{\frac{d}{2}} g(\zeta u)} \notag \\ & =
   \frac{1}{\zeta^{\frac{d}{2}+1}} \Bigg \{ \lim_{u \rightarrow 0} \bigg(
   \kappa'(u) + \frac{d}{2} \frac{\kappa(u)}{u} \bigg)
   \frac{g(u)}{g(\zeta u)} + \lim_{u \rightarrow 0}
   \frac{\kappa(u)g'(u)}{g(\zeta u)} \Bigg \} \notag \\ & =
   \frac{l_1+l_2}{\zeta^{\frac{d}{2}+1}} \ge 0 \;.
 \end{align}
 Put $l_3 = (l_1+l_2)/\zeta^{\frac{d}{2}+1}$. It follows from
 \eqref{eq:L1} that for any $\eta \in (0,\eta_0)$, there exists
 $0<\eta_1<\eta$ such that
\[
\frac{u^{\frac{d}{2}}g(u)}{\int_0^{\zeta u} v^{\frac{d}{2}} g(v) \, dv} \le 
\frac{l_3+1}{\kappa(u)}
\]
whenever $u \in (0,\eta_1)$.  Then, since $\kappa \in \mathcal{K}$,
there exists $\eta_2 \in (0,\eta_1)$ such that
\[
\int_0^{\eta_2} \frac{1}{\kappa(u)} \, du < \infty \;.
\]
Furthermore, since $g$ is continuous on $[\eta_2,\eta]$,
$\int_{\eta_2}^\eta u^{\frac{d}{2}} g(u) \, du < \infty$.  Putting all
of this together, we have for any $\eta < \eta_0$,
\begin{align*}
  \int_{0}^{\eta} \frac{u^{\frac{d}{2}} g(u)}{\int_0^{\zeta u}
    v^{\frac{d}{2}} g(v) \, dv} \, du & = \int_{0}^{\eta_2}
  \frac{u^{\frac{d}{2}} g(u)}{\int_0^{\zeta u} v^{\frac{d}{2}} g(v) \, dv}
  \, du + \int_{\eta_2}^{\eta} \frac{u^{\frac{d}{2}} g(u)}{\int_0^{\zeta u}
    v^{\frac{d}{2}} g(v) \, dv} \, du \\ & \le \int_{0}^{\eta_2}
  \frac{l_3+1}{\kappa(u)} \, du + \frac{\int_{\eta_2}^{\eta}
    u^{\frac{d}{2}} g(u) \, du}{\int_0^{\zeta\eta_2} v^{\frac{d}{2}} g(v)
    \, dv} \\ & < \infty \;.
\end{align*}
\end{proof}

\begin{proof}[Proof of Theorem~\ref{thm:main}]
  Assume that $h$ is zero near the origin, and define $\eta_0 = \sup
  \big \{\eta \in \mathbb{R}_+: \int_0^\eta u^{\frac{d}{2}}h(u) \, du
  = 0 \big \}$.  Clearly, $J := \int_0^{\frac{3\eta_0}{2}}
  u^{\frac{d}{2}} h(u) \, du >0$.  Now, for $s \in [0,\infty)$, we
  have
\[
\int_\mathbb{R_+} v^{\frac{d}{2}} e^{-\frac{sv}{2}} h(v) \, dv \ge
\int_0^{\frac{3\eta_0}{2}} v^{\frac{d}{2}} e^{-\frac{sv}{2}} h(v) \,
dv \ge J e^{-\frac{3 \eta_0s}{4}} \;.
\]
Therefore, for $u \in \mathbb{R}_+$ and $s \in [0,\infty)$, we have
\[
\psi(u;s) = \frac{u^{\frac{d}{2}} \, e^{ -\frac{s u}{2}} \,
  h(u)}{\int_{\mathbb{R}_+} v^{\frac{d}{2}} \, e^{ -\frac{s v}{2}} \,
  h(v) \, dv} \le J^{-1} u^{\frac{d}{2}} h(u) \, e^{-\frac{s u}{2} +
  \frac{3 \eta_0s}{4}} \;.
\]
Now, by considering $u \ge \eta_0$ and $u < \eta_0$ separately, we can
see that
\[
\psi(u;s) \le J^{-1} u^{\frac{d}{2}} h(u) \, e^{\frac{s u}{4}}
\]
for all $u \in \mathbb{R}_+$ and all $s \in [0,\infty)$.  Hence,
\eqref{eq:main} of Lemma~\ref{lem:suff} holds with $\zeta=3/2$ and $\nu(u)
= J^{-1}u^{\frac{d}{2}} h(u)$, so the result follows.

We now prove that the second condition is sufficient.  Assume that
there exists $g \in \mathcal{C}(\kappa,\zeta)$ such that $\lim_{u
  \rightarrow 0} \frac{h(u)}{g(u)} = l \in \mathbb{R}_+$.  Then by
Lemma~\ref{lem:suff3}, there exists $\eta>0$ such that
\[
\int_0^\eta \frac{u^{\frac{d}{2}} \, g(u)}{\int_0^{\zeta u} v^{\frac{d}{2}}
  \, g(v) \, dv} \, du < \infty \;,
\]
and such that
\[
\frac{l}{2} \le \frac{h(u)}{g(u)} \le 2l
\]
whenever $u \in (0,\zeta\eta)$.  It follows that
\[
\int_0^\eta \frac{u^{\frac{d}{2}} \, h(u)}{\int_0^{\zeta u} v^{\frac{d}{2}}
  \, h(v) \, dv} \, du \le 4 \int_0^\eta \frac{u^{\frac{d}{2}} \,
  g(u)}{\int_0^{\zeta u} v^{\frac{d}{2}} \, g(v) \, dv} \, du \;,
\]
and the result follows from Lemma~\ref{lem:suff2}.

Finally, we prove that the third condition is sufficient.  Note first
that we may assume that $h$ is not zero near the origin, since,
otherwise, the result follows immediately from condition (1).  Assume
that there exists $g \in \mathcal{C}(\kappa,\zeta)$ such that
$u^{\frac{d}{2}} g(u)$ and $\frac{h(u)}{g(u)}$ are both non-decreasing
near the origin.  By Lemma~\ref{lem:suff3}, there exists $\eta'>0$
such that
\[
\int_0^{\eta'} \frac{u^{\frac{d}{2}} \, g(u)}{\int_0^{\zeta u}
  v^{\frac{d}{2}} \, g(v) \, dv} \, du < \infty \;.
\]
Now let $\eta \in (0, \eta')$ be such that $g$ and $h$ are both
strictly positive for $u \in (0,\eta)$, and $u^{\frac{d}{2}} g(u)$ and
$\frac{h(u)}{g(u)}$ are both non-decreasing in that interval.  For $u
\in (0,\eta)$, let $t(u) = h(u)/g(u)$.  For any $u \in (0,\eta/\zeta)$, we
have
\begin{equation}
  \label{eq:pf1}
  \int_u^{\zeta u} v^{\frac{d}{2}} \, h(v) \, dv \ge \int_u^{\zeta u}
  v^{\frac{d}{2}} \, t(u) g(v) \, dv  \;,
\end{equation}
since $t$ is non-decreasing.  If $v \in [u,\zeta u]$, then $v \ge
\frac{\zeta(v-u)}{\zeta-1} \ge 0$.  For any $u \in (0,\eta/\zeta)$, we have
\begin{equation}
  \label{eq:pf2}
  \int_u^{\zeta u} v^{\frac{d}{2}} \, g(v) \, dv \ge \int_u^{\zeta u} \Big[
  \frac{\zeta(v-u)}{\zeta-1} \Big]^{\frac{d}{2}} \, g \Big( \frac{\zeta(v-u)}{\zeta-1}
  \Big) \, dv = \frac{\zeta-1}{\zeta} \int_0^{\zeta u} w^{\frac{d}{2}} \, g(w) \,
  dw\;.
\end{equation}
It follows from \eqref{eq:pf1} and \eqref{eq:pf2} that, for $u \in
(0,\eta/\zeta)$, we have
\[
\frac{u^{\frac{d}{2}} \, h(u)}{\int_0^{\zeta u} v^{\frac{d}{2}} \, h(v) \,
  dv} \le \frac{u^{\frac{d}{2}} \, h(u)}{\int_u^{\zeta u} v^{\frac{d}{2}}
  \, h(v) \, dv} \le \frac{u^{\frac{d}{2}} \, t(u) g(u)}{\big(
  \frac{\zeta-1}{\zeta} \big) t(u) \int_0^{\zeta u} v^{\frac{d}{2}} \, g(v) \, dv}
= \frac{\zeta}{\zeta-1} \frac{u^{\frac{d}{2}} \, g(u)}{ \int_0^{\zeta u}
  v^{\frac{d}{2}} \, g(v) \, dv} \;.
\]
Hence,
\[
\int_0^{\frac{\eta}{\zeta}} \frac{u^{\frac{d}{2}} \, h(u)}{\int_0^{\zeta u}
  v^{\frac{d}{2}} \, h(v) \, dv} \, du \le \frac{\zeta}{\zeta-1}
\int_0^{\frac{\eta}{\zeta}} \frac{u^{\frac{d}{2}} \, g(u)}{ \int_0^{\zeta u}
  v^{\frac{d}{2}} \, g(v) \, dv} \, du \le \frac{\zeta}{\zeta-1}
\int_0^{\eta'} \frac{u^{\frac{d}{2}} \, g(u)}{ \int_0^{\zeta u}
  v^{\frac{d}{2}} \, g(v) \, dv} \, du < \infty \;,
\]
and the result follows from Lemma~\ref{lem:suff2}.
\end{proof}

\section{Mixing densities that are faster than polynomial near the
  origin}
\label{sec:ftp}

In this section, we provide details to back the claims made in the
Introduction.  We begin by using Corollary~\ref{cor:main} to show
that, if $h$ is log-normal, generalized inverse Gaussian, or
Fr\'{e}chet, then the DA Markov operator is trace-class.  We then
provide an example of a mixing density that is faster than polynomial
near the origin, but for which \eqref{eq:fail} holds.  Again, this
shows that our result is not applicable to this mixing density.

Let $h$ be a $\mbox{GIG}(v,\alpha,\gamma)$ density, so that
\begin{equation*}
  h(u) = b \, u^{v-1} \exp \Big \{ -
  \frac{1}{2} \Big(\alpha u + \frac{\gamma}{u} \Big) \Big \}
  I_{\mathbb{R}_+}(u) \;,
\end{equation*}
where $\alpha,\gamma \in \mathbb{R}_+$, $v \in \mathbb{R}$, and
$b=b(v,\alpha,\gamma)$ is the normalizing constant.  It's easy to see
that conditions ($S3$) \& ($S4$) hold for all members of this family.
Taking $\rho=1$ and $\tau=v-1$ in Corollary~\ref{cor:main}, we have
\begin{align*}
  \frac{d}{du} \frac{h(u)}{g_{\rho,\tau}(u)} = & b \frac{d}{du} \exp
  \Big\{ -\frac{\alpha u}{2} - \frac{\gamma}{2u} + (\log u)^2 \Big\}
  \\ & = b \Big( -\frac{\alpha}{2} + \frac{\gamma}{2u^2} + \frac{2\log
    u}{u} \Big) \exp \Big\{ -\frac{\alpha u}{2} - \frac{\gamma}{2u} +
  (\log u)^2 \Big\} \;,
\end{align*}
which is clearly non-negative in a neighborhood of 0.  Thus, $K$ is
trace-class.

Suppose $h$ is a $\mbox{Fr\'{e}chet}(\alpha,\gamma)$ density, i.e., 
\begin{equation*}
  h(u) = b \, u^{-(\alpha+1)} \, e^{-\frac{\gamma^\alpha}{u^\alpha}} \, 
  I_{\mathbb{R}_+}(u) \;,
\end{equation*}
where $\alpha,\gamma>0$, and $b=b(\alpha,\gamma)$ is the normalizing
constant.  Assume that $\alpha>d/2$ so that condition ($S3$) holds.
Taking $\rho=1$ and $\tau=-(\alpha+1)$, we have
\[
\frac{d}{du} \frac{h(u)}{g_{\rho,\tau}(u)} = b \frac{d}{du} \exp
\Big\{ -\frac{\gamma^\alpha}{u^\alpha} + (\log u)^2 \Big\} = b \bigg(
\frac{\alpha \gamma^\alpha}{u^{\alpha+1}} + \frac{2 \log u}{u} \bigg)
\exp \Big\{ -\frac{\gamma^\alpha}{u^\alpha} + (\log u)^2 \Big\} \;,
\]
which is clearly non-negative in a neighborhood of 0, so $K$ is
trace-class.

Finally, let $h$ be a $\mbox{Log-normal}(\mu,\gamma)$ density, so that
\begin{equation*}
  h(u) = \frac{b}{u} \exp \Big \{ -\frac{1}{2 \gamma} \Big( \log u -
  \mu)^2 \Big) \Big \} I_{\mathbb{R}_+}(u) \;,
\end{equation*}
where $\gamma>0$, $\mu \in \mathbb{R}$ and $b=b(\gamma,\mu)$ is the
normalizing constant.  Every member of this family satisfies
conditions ($S3$) \& ($S4$).  Taking $\rho = \frac{1}{2 \gamma} $ and
$\tau = \frac{\mu}{\gamma}-1$, we have
\[
\frac{h(u)}{g_{\rho,\tau}(u)} = b \, e^{-\frac{\mu^2}{2\gamma}} \;,
\]
and the result follows.

We end this section by showing that there exist mixing densities that
are faster than polynomial near the origin, but are not in the domain
of application of Theorem~\ref{thm:main}.  Consider the following
mixing density
\[
h(u) = b \exp \Big \{ (\log u) \log(-\log u) - \Big( \frac{d}{2} +1
\Big) \log u \Big \} I_{(0,1)}(u) \;,
\]
where $b=b(d)$ is the normalizing constant.  For any real $c$, we have
\begin{align*}
  \frac{d}{du} \frac{h(u)}{u^c} & = \bigg \{ \frac{d}{du} \bigg[ (\log
  u) \log(-\log u) - \Big( c + \frac{d}{2} +1 \Big) \log u \bigg]
  \bigg \} \frac{h(u)}{u^c} \\ & = \bigg[ \frac{1}{u} \log(-\log u) -
  \frac{( d + 2c)}{2u} \bigg] \frac{h(u)}{u^c} \;.
\end{align*}
When $u>0$ is small, $\frac{d}{du} \frac{h(u)}{u^c}>0$, so $h(u)$ is
indeed faster than polynomial near the origin.  We now show that
\eqref{eq:fail} holds.  Define
\[
\nu_h(u) = \frac{u^{\frac{d}{2}} h(u)}{\int_0^{2u} v^{\frac{d}{2}}
  h(v) \, dv} \;,
\]
and let $\phi(u) = -u (\log u) I_{(0,1)}(u)$.  An application of
L'H\^{o}pital's rule yields
\begin{align*}
  \lim_{u \rightarrow 0} \phi(u) \nu_h(u) & = \lim_{u \rightarrow 0}
  \frac{ \frac{d}{du} \big[ \phi(u) u^{\frac{d}{2}} h(u)
    \big]}{2(2u)^{\frac{d}{2}} h(2u)} \\ & = \lim_{u \rightarrow 0}
    \bigg[ \phi'(u) + \frac{\phi(u) \log(-\log u)}{u} \bigg]
    \frac{u^{\frac{d}{2}} h(u)}{2(2u)^{\frac{d}{2}} h(2u)} \;.
\end{align*}
Now
\[
\phi'(u) + \frac{\phi(u) \log(-\log u)}{u} = -(\log u) \log(-\log
u) -\log u - 1 \;,
\]
and
\begin{align*}
  \frac{u^{\frac{d}{2}} h(u)}{2(2u)^{\frac{d}{2}} h(2u)} & =
  \frac{1}{2} \exp \Big \{ (\log u) \log(-\log u) -\log u - \log[-\log
  (2u)] \log (2u) + \log (2u) \Big \} \\ & = \exp \bigg \{ -(\log 2)
  \log[-\log (2u)] + (\log u) \log \frac{\log u}{\log u + \log 2}
  \bigg \} \\ & = \frac{\exp \big \{ (\log u) \log \frac{\log u}{\log
      u + \log 2} \big \}}{(-\log 2 -\log u)^{\log 2}} \;.
\end{align*}
Thus,
\[
\lim_{u \rightarrow 0} \phi(u) \nu_h(u) = \lim_{u \rightarrow 0}
\frac{-(\log u) \log(-\log u) -\log u - 1}{(-\log 2 -\log u)^{\log 2}}
\exp \bigg \{ (\log u) \log \frac{\log u}{\log u + \log 2} \bigg \} \;.
\]
It is straightforward to show that
\[
\lim_{u \rightarrow 0} \, (\log u) \log \frac{\log u}{\log u + \log 2} =
- \log 2 \;,
\]
and that
\[
\lim_{u \rightarrow 0} \frac{-(\log u) \log(-\log u) -\log u -
  1}{(-\log 2 -\log u)^{\log 2}} = \infty \;.
\]
Therefore, 
\[
\lim_{u \rightarrow 0} \phi(u) \nu_h(u) = \infty \;.
\]
It follows that, for all $\eta$ in a small neighborhood of the origin,
we have
\[
\int_0^\eta \nu_h(u) \, du = \int_0^\eta \frac{\phi(u)
  \nu_h(u)}{\phi(u)} \, du \ge \int_0^\eta \frac{1}{\phi(u)} \, du =
\infty \;.
\]

\section{The Haar PX-DA algorithm}
\label{sec:Haar}

The Haar PX-DA algorithm always exists in the special case where $a =
\frac{d+1}{2}$, but, outside of this case, its existence requires an
additional regularity condition.  Indeed, to define the Haar PX-DA
algorithm, we must assume that
\begin{equation}
  \label{eq:H}
  \int_0^\infty t^{n+\frac{(d+1-2a)d}{2}-1} \, \Bigg[ \prod_{i=1}^n h(t z_i) \Bigg] \,
  dt < \infty
\end{equation}
for (almost) all $z \in \mathbb{R}_+^n$.  HJK\&Q show that
\eqref{eq:H} holds if
\begin{equation}
  \label{eq:Ha}
\int_0^\infty u^{\frac{(d+1-2a)d}{2}} \, h(u) \, du < \infty \;.
\end{equation}
Note that \eqref{eq:Ha} is automatic when $a = \frac{d+1}{2}$.  Now
assume that \eqref{eq:H} holds, and define another parametric family
of univariate density functions given by
\[
e(v;z) = \frac{v^{n-1} \, \prod_{i=1}^n h(v z_i) \,
  I_{\mathbb{R}_+}(v)}{\int_0^\infty t^{n-1} \, \prod_{i=1}^n h(t z_i)
  \, dt} \;.
\]
Let $\Phi^* = \{(\beta^*_m,\Sigma^*_m)\}_{m=0}^\infty$ denote the Haar
PX-DA Markov chain.  If the current state of the chain is $(\beta^*_m,
\Sigma^*_m) = (\beta,\Sigma)$, then we simulate the new state,
$(\beta^*_{m+1}, \Sigma^*_{m+1})$, using the following four-step
procedure.

\baro \vspace*{2mm}
\noindent {\rm Iteration $m+1$ of the Haar PX-DA algorithm:}
\begin{enumerate}
\item Draw $\{Z'_i\}_{i=1}^n$ independently with $Z'_i \sim \psi \Big(
  \cdot \;; \big( \beta^T x_i - y_i \big)^T \Sigma^{-1} \big( \beta^T x_i -
  y_i \big) \Big)$, and call the result $z' = (z'_1,\dots,z'_n)$.
\item Draw $V \sim e(\cdot \, ;z')$, call the result $v$, and set $z =
  (vz'_1,\dots,vz'_n)^T$.
\item Draw
\[
\Sigma^*_{m+1} \sim \mbox{IW}_d \bigg( n-p+2a-d-1, \Big( y^T Q^{-1} y -
  \mu^T \Omega^{-1} \mu \Big)^{-1} \bigg)  \;.
\]
\item Draw $\beta^*_{m+1} \sim \mbox{N}_{p,d} \big( \mu, \Omega,
  \Sigma^*_{m+1} \big)$ \vspace*{-2.5mm}
\end{enumerate}
\barba

The following result is a direct consequence of Theorem~\ref{thm:main}
and Theorems 1 and 2 from \citet{khar:hobe:2011}.

\begin{corollary}
  \label{cor:Haar}
  Let $h$ be a mixing density such that \eqref{eq:H} holds.  Let $K$
  and $K^*$ denote the Markov operators associated with the DA and
  Haar PX-DA Markov chains, respectively.  Assume that one of the
  three conditions of Theorem~\ref{thm:main} holds.  Then $K^*$ is
  trace-class.  Moreover, letting $\{\lambda_i\}_{i=1}^\infty$ and
  $\{\lambda^*_i\}_{i=1}^\infty$ denote the ordered eigenvalues of $K$
  and $K^*$, respectively, we have that $0 \le \lambda^*_i \le
  \lambda_i < 1$ for all $i \in \mathbb{N}$, and $\lambda^*_i <
  \lambda_i$ for at least one $i \in \mathbb{N}$.
\end{corollary}

\vspace*{5mm}

\noindent {\bf \large Acknowledgment}.  The second author was
supported by NSF Grant DMS-15-11945.

\vspace*{8mm}

\noindent {\LARGE \bf Appendices}
\begin{appendix}

\vspace*{-3mm}

\section{Matrix Normal and Inverse Wishart Densities}
\label{app:dist}

\begin{description}
\item[Matrix Normal Distribution] Suppose $Z$ is an $r \times c$
  random matrix with density
\[
f_{Z}(z) = \frac{1}{(2\pi)^{\frac{rc}{2}} |A|^{\frac{c}{2}}
  |B|^{\frac{r}{2}}} \exp \bigg[ -\frac{1}{2}\mbox{tr} \Big\{ A^{-1}(z
  - \theta) B^{-1} (z - \theta)^T \Big\} \bigg] \;,
\]
where $\theta$ is an $r \times c$ matrix, $A$ and $B$ are $r \times r$
and $c \times c$ positive definite matrices.  Then $Z$ is said to have
a \textit{matrix normal distribution} and we denote this by $Z \sim
\mbox{N}_{r,c} (\theta,A,B)$ \citep[][Chapter 17]{arno:1981}.

\item[Inverse Wishart Distribution] Suppose $W$ is an $r \times r$
  random positive definite matrix with density
\[
f_{W}(w) = \frac{|w|^{-\frac{m+r+1}{2}} \exp \Big \{ -\frac{1}{2}
  \mbox{tr} \big( \Theta^{-1} w^{-1} \big) \Big\}}{ 2^{\frac{mr}{2}}
  \pi^{\frac{r(r-1)}{4}} |\Theta|^{\frac{m}{2}} \prod_{i=1}^r \Gamma
  \big( \frac{1}{2}(m+1-i) \big)} I_{{\cal S}_r}(W) \;,
\]
where $m > r-1$ and $\Theta$ is an $r \times r$ positive definite
matrix.  Then $W$ is said to have an \textit{inverse Wishart
  distribution} and this is denoted by $W \sim \mbox{IW}_r(m,
\Theta)$.
\end{description}

\section{Proof of Proposition~\ref{prop:pf}}
\label{app:proof}

\begin{proof}[Proof of Proposition~\ref{prop:pf}]
  Fix $(\rho,\tau) \in \mathbb{R}_+ \times \mathbb{R}$.  First, that
  $u^{\frac{d}{2}} g_{\rho,\tau}(u)$ is non-decreasing in a
  neighborhood of the origin is obvious. To show that $g_{\rho,\tau}
  \in \mathcal{C}(\kappa,3/2)$ with $\kappa(u) = u (\log u)^2$, we
  demonstrate that $g_{\rho,\tau}$ satisfies the three conditions that
  define $\mathcal{C}(\kappa,3/2)$. Clearly $\lim\limits_{u
    \rightarrow 0} g_{\rho,\tau}(u) = 0$, hence $u^{\frac{d}{2}}
  g_{\rho,\tau}(u)$ is bounded in a neighborhood of the
  origin. Moreover,
  \[
\lim\limits_{u \rightarrow 0} \kappa(u) u^{\frac{d}{2}}
g_{\rho,\tau}(u) =
\lim\limits_{u \rightarrow 0} u(\log u)^2 u^{\frac{d}{2}}
g_{\rho,\tau}(u) = 0.
\]
Now, note that
\begin{align*}
  & \lim_{u \rightarrow 0} \Big( \kappa'(u) +
  \frac{d}{2}\frac{\kappa(u)}{u}
  \Big) \frac{g_{\rho,\tau}(u)}{g_{\rho,\tau}(3u/2)} \\
  & = \lim_{u \rightarrow 0} \Big[ \frac{d+2}{2} (\log u)^2 + 2\log u
  \Big] \frac{g_{\rho,\tau}(u)}{g_{\rho,\tau}(3u/2)} \\
  & = \lim_{u \rightarrow 0} \Big[ \frac{d+2}{2} + \frac{2}{\log u}
  \Big] \lim_{u \rightarrow 0} \bigg[ (\log u)^2 \exp \Big\{ \rho
  \Big(\log \frac{3u}{2}\Big)^2 - \tau \log\frac{3u}{2} - \rho(\log
  u)^2 + \tau \log u \Big\} \bigg] \\ & = \frac{d+2}{2} \exp \Big\{
  \log \frac{3}{2} \Big( \rho \log \frac{3}{2} - \tau \Big) \Big\}
  \lim_{u \rightarrow 0} \Big[ (\log u)^2 u^{2\rho\log\frac{3}{2}}
  \Big] \\ & = 0 \;,
\end{align*}
and
\begin{align*}
  \lim_{u \rightarrow 0} \frac{\kappa(u)
    g'_{\rho,\tau}(u)}{g_{\rho,\tau}(3u/2)} &= \lim_{u \rightarrow 0}
  \frac{u(\log u)^2 g'_{\rho,\tau}(u)}{g_{\rho,\tau}(3u/2)}\\
  & = \lim_{u \rightarrow 0} \bigg[ u (\log u)^2 \Big( \frac{\tau -
    2\rho \log u}{u} \Big )
  \frac{g_{\rho,\tau}(u)}{g_{\rho,\tau}(3u/2)} \bigg] \\ & = \exp
  \Big\{ \log \frac{3}{2} \Big( \rho \log \frac{3}{2} - \tau \Big)
  \Big\} \lim_{u \rightarrow 0} \Big[ ( \tau - 2\rho \log u) (\log
  u)^2 u^{2\rho\log\frac{3}{2}} \Big] \\ & = 0 \;.
\end{align*}
It follows that \eqref{eq:twocs} holds with $\zeta=3/2$ and $l_1=l_2=0$.
Thus $g_{\rho,\tau} \in \mathcal{C}(\kappa,3/2)$.
\end{proof}

\end{appendix}

\bibliographystyle{ims}
\bibliography{refs}

\end{document}